\documentclass[preprint,1p]{elsarticle}
\usepackage{amssymb}
\usepackage{amsthm}
\usepackage{graphicx}
\usepackage{latexsym}                                                                      
\usepackage{amsmath,amsthm,amsfonts}                                                                                                                   
\usepackage{palatino}                                                                       
\usepackage[mathscr]{eucal}

\newtheorem{thm}{Theorem}[section]

 \newtheorem{lem}[thm]{Lemma}
\newtheorem{cor}[thm]{Corollary}
\theoremstyle{definition}

\theoremstyle{remark}
\newtheorem{remark}[thm]{Remark}
\numberwithin{equation}{section}
\newcommand{\R}{{\Bbb R}}

\newcommand{\C}{{\Bbb C}}

\begin{document}
\begin{frontmatter}

\title{On the uniqueness of semi-wavefronts for  non-local delayed reaction-diffusion equations }
\author{ Maitere Aguerrea}
\ead{maguerrea@ucm.cl}

\address{Facultad de Ciencias B\'asicas, Universidad Cat\'olica del Maule,\\ Casilla 617,
Talca, Chile.}

\begin{abstract}
\noindent We establish  the uniqueness of semi-wavefront solution for  a non-local delayed reaction-diffusion
equation.  This result is obtained by using a generalization  of the Diekman-Kaper theory for a nonlinear convolution equation. Several applications to the systems of non-local reaction-diffusion equations with distributed time  delay are also considered.

 \end{abstract}

\begin{keyword}
time-delayed reaction-diffusion equation; positive wavefront; non-local interaction; minimal wave; semi-wavefront; uniqueness.\\
\end{keyword}
\end{frontmatter}

\section{Introduction.}\label{into}
The main object of study in this work is  the non-local reaction-diffusion equation
\begin{equation}\label{i1}
u_t(t,x) =  u_{xx}(t,x)- f(u(t,x)) +  \int_0^\infty\int_{\R}K(s,w)g(u(t-s,x-w))dwds,\end{equation}
where  the time $t\geq 0$, $x\in \R$, the kernel $K$  satisfies $K\in L^1(\R_+\times\R)$, $K\geq 0$ and  $\int_0^\infty\int_{\R}K(s,w)dwds=1$. Here, $K$ can be asymmetric. We also assume the following conditions on the monostable nonlinearity $g$ and the function $f$:
 \begin{description}
 \item[\bf{ $H_1$:}] $g \in C(\R_+,\R_+)$ is  such that $g(0)=0$,  $g(s)>0$ for all $s>0$, and differentiable at 0 with $g'(0)>0$. 
\item[\bf{ $H_2$:}] $f \in C^1(\R_+,\R_+)$, $f(0)=0$, is strictly increasing with $f'(0)<g'(0)$.
\item[\bf{ $H_3$:}] 
$g,f \in C^{1,\alpha}$ in some neighborhood of $0$, with  $\alpha\in(0,1)$.
\end{description}
Equation (\ref{i1}), with appropriate  $f,g$ and $K$,  is often used to model  ecological and biological processes where the typical interpretation of $u(t,x)$  is the population density of mature species (see, e.g. \cite{Brit,fz,FWZ,gouk,gouS,gouss,lrw,MSo,OG,swz,tz}). In the particular case,  when $K(s,w)=\delta(s-h)K(w)$,  equation (\ref{i1}) reduces to the well studied model  
\begin{equation}\label{i1a}
u_t(t,x) =  u_{xx}(t,x)- f(u(t,x)) +  \int_{\R}K(w)g(u(t-h,x-w))dw.
\end{equation}
We are interesting  in the study of semi-wavefront solutions of equation  (\ref{i1}), i.e.  bounded positive continuous non-constant waves  $u(t,x) = \phi(x +ct)$, propagating with speed $c$,  and  satisfying the boundary condition  $\phi(-\infty) = 0$. An important special case of  semi-wavefront is a wavefront, i.e.  positive classical solution $u(t,x) = \phi(x +ct)$ satisfying $\phi(-\infty) = 0$ and $\phi(+\infty) = \kappa$. 

During the last decade, the existence and uniqueness of  traveling wave solutions for equation  (\ref{i1}) have been  investigated in several papers.  For instance, the existence problem has been approached by means of different methods and assuming different conditions on $f$, $K$ and $g$ (see, \cite{map,ft,fhw,gpt,LiWu,ma1,tat,wlr2,wang,xw,ycw}). Surprisingly,  the uniqueness question  appears to be considerably more difficult to answer
than the existence question.  In fact, only few theoretical studies have considered this important problem. Let us mention here \cite{agt,atv,Ai,fz,tz,ttt,wlr2,wl}, where the uniqueness of  semi-wavefronts for (\ref{i1}), was proved only in special  cases,  and  almost always  assuming condition
\begin{equation}\label{lip}
 |g(s_{1})-g(s_{2})|\leq L|s_{1}-s_{2}|, \,\,\, s_{1}, s_{2}\geq0,\,\, \text{for some}\,\,  L>0,
 \end{equation} 
 with $L=g'(0)$. It is worthwhile mentioning that the main idea of   the proofs of uniqueness in \cite{fz,tz,wl} is due to the seminal paper \cite{dk} by Diekmann and Kaper, where   it requires Lipschitz condition (\ref{lip}) with $L=g'(0)$. This condition is essential in constructions \cite{dk,fz,tz,wl} and can not be omitted or weakened within the framework of \cite{wl}. On the other hand,  works \cite{agt} and \cite{atv} showed that the  assumption  (\ref{lip}) with $L=g'(0)$ is not necessary to obtain the uniqueness  of fast wave solution of (\ref{i1a}) when $h>0$, both in the local and non-local cases. 

In any case, for the  non-local reaction-diffusion equation with distributed delay (\ref{i1}), the uniqueness problem of the semi-wavefront is still unsolved in the general case (e.g. $K$ asymmetric and $L\not=g'(0)$). In particular, the uniqueness problem of the minimal wave to (\ref{i1}), has not yet been solved (see \cite{fz}). The main objective of this work is to present a solution of this open problem. We also weaken  or  remove some restrictions on kernel and nonlinearities. 

In order to apply the techniques of \cite{agt}, we  must rewrite the equation (\ref{i1}) as the scalar integral equation
\begin{equation*}\label{dk1} 
 \phi(t) = \int_Xd\rho(\tau)\int_\R \mathcal N(s,\tau)g(\phi(t-s),\tau) ds, \quad t \in \R.
\end{equation*}
Here $(X,\rho)$ denotes a space with finite measure $\rho$, $\mathcal N(s,\tau) \geq 0$ is
integrable on $\R\times X$ with  $\int_{\R}\mathcal N(s,\tau)ds>0$,  $\tau \in X,$ while measurable $g: \R_+ \times X \to \R_+,$ $g(0,\tau) \equiv 0, $ is continuous 
in $\phi$ for every fixed $\tau \in X$.
We apply the theory developed in \cite{agt}  to prove the  uniqueness (up to translation) of  semi-wavefronts.
Since our main focus here is  the uniqueness of  semi-wavefronts, we assume  the existence of a semi-wavefront to (\ref{i1}).

Before presenting  our results, we have to introduce several definitions. Let  $c_*$ [respectively, $c_\star$] be the minimal value of  $c$ for which
\begin{equation*}\label{c1}
\chi_0(z,c):=z^2- cz-f'(0)+g'(0)\int_0^\infty \int_{\R}K(s,w)e^{-z\left(cs+w\right)}dwds,
\end{equation*}
\vspace{-0.2cm}
[respectively,
\vspace{-0.1cm}
\begin{equation*}\label{c2}
\chi_L(z,c):=z^2- cz-\inf_{s\geq 0}f'(s)+L\int_0^\infty \int_{\R}K(s,w)e^{-z\left(cs+w\right)}dwds,\, L\geq g'(0)] 
\end{equation*} 
 has  at least one positive root. We observe that  $c_\star\geq c_*$ and the function $\chi_0(z,c)$ is associated with the  linearization of (\ref{i1}) along the trivial equilibrium.  Some estimates for $c_*, c_\star$ can be found in   \cite{av,tt,WM}. For example,
 if $\int_0^\infty \int_{\R}K(s,w)w\,dwds\leq 0,$ then  
 \begin{align}\label{speed}c_*>\frac{g'(0)\left|\int_0^\infty \int_{\R}K(s,w)w\,dwds\right|}{1+g'(0)\int_0^\infty \int_{\R}K(s,w)s\,dwds}>0.
 \end{align}

 Let us present now the  main results of the paper, they  follow from  more general theorem  which will be proved in Section \ref{main}. 
  
\begin{thm} \label{2}Assume that {\bf{ $H_1$}}-{\bf{ $H_3$}} hold. Suppose further that for any $c\in \R$, there exists some $\gamma^\#=\gamma^\#(c)\in (0,+\infty]$ such that  $\chi_0(z,c)<\infty$ for each $z \in [0,\gamma^\#)$ and diverges, if $z>\gamma^\#$. If   $g$ satisfies the condition (\ref{lip}), then  equation (\ref{i1}) has at most one (modulo translation)  semi-wavefront   solution $u(x,t)=\phi(x+ct)$ for each $c> c_\star$, if $ \chi_L(\gamma^\#(c_\star)-,c_\star)=0$, and  for each   $c\geq c_\star$, if $ \chi_L(\gamma^\#(c_\star)-,c_\star)\not=0$. 
\end{thm}

\begin{thm} \label{3}Assume all the conditions of Theorem \ref{2}.  Then for any  $c<c_*$,  the equation (\ref{i1}) has no semi-wavefront solution  $u(x,t)=\phi(x+ct)$ propagating  with speed  $c$. 
\end{thm}

\begin{remark} We observe  that Theorem \ref{2} shows that  the special  Lipschitz condition $|g(s)-g(t)|\leq g'(0)|s-t|$ is not necessary to prove the uniqueness of fast semi-wavefronts solution. Moreover, our result also incorporates the critical case when $L=g'(0)$ and asymmetric kernels. Thus Theorem \ref{2} improves  the uniqueness results in \cite{fz,tz}, where  the uniqueness was established  for  $g$ satisfying  (\ref{lip})  with $L=g'(0)$  and  assuming either even or Gaussian kernel $K$, and without considering (with mentioned properties)   the uniqueness of the critical semi-wavefront.  
\end{remark}

\begin{remark} 
We also observe that the existence of $\gamma^\#$ in Theorem \ref{2}  is  a strong restriction.  In fact, we will show in the following section that the existence of semi-wavefront with speed $c$ implies that  $\chi_0(\gamma,c)<\infty$ for some $\gamma>0$. 
\end{remark}
The paper is organized as follows.  
Section 2 contains some preliminary results and transformations need  to apply  the method of \cite{agt} (for the convenience of the reader, we briefly  describe this method in Appendix). In  
Section 3, we analyze the characteristic equations  $\chi_0(z,c)=0$ and  $\chi_L(z,c)=0$. The estimation (\ref{speed}) for $c_\star$ is proved there. In Section 4, we prove our main results. Finally, in the last section, the uniqueness theorem is applied   to several population and epidemic models.

\section{Preliminaries.}
 
It is  clear that the profile $y=\phi$ of the  semi-wavefront  solution $u(t,x) = \phi(x +ct)$ to (\ref{i1}) must satisfy the equation
\begin{equation}\label{i2}
 y''(t) - cy'(t)-f(y(t))+ \int_0^\infty\int_{\R}K(s,w)g\left(y(t-cs-w)\right)dwds=0
\end{equation}
for all $t \in \R$. Note that this equation  can be written as 
\begin{equation}\label{ii3}
 y''(t) - cy'(t)-\beta y(t)+f_\beta(y(t))+  \int_0^\infty\int_{\R}K(s,w)g(y(t-cs-
w))dwds=0,
\end{equation}
where  $f_\beta(s)=\beta s- f(s)$ for some  $\beta>f'(0)$.  Hence, in order to establish the uniqueness of semi-wavefront solution to (\ref{i1}), we have to prove the uniqueness of positive bounded solution $\phi$  of equation (\ref{i2}), satisfying $\phi(-\infty)=0$.  

Being  $\phi$  a positive bounded solution to (\ref{i2}),  it should satisfy the  integral equation
\begin{align*} 
\phi(t)&=\frac{1}{\sigma(c)}\left(
\int_{-\infty}^te^{\nu(c)(t-s)}(\mathcal G\phi)(s)ds
+\int_t^{+\infty}e^{\mu(c)(t-s)}(\mathcal G\phi)(s)ds\right)\\\nonumber
&= \int_\R k_1(t-s)(\mathcal G\phi)(s)ds,\,\, t\in \R,
\end{align*}
where $$k_1(s)=(\sigma(c))^{-1}\left\{\begin{array}{cc}e^{\nu(c) s}, & s\geq 0 \\e^{\mu(c) s}, & s<0\end{array}\right.,$$ $\sigma(c)=\sqrt{c^2+4\beta}$,  $\nu(c)<0<\mu(c)$ are the roots of  $z^2-cz-\beta=0$
 and  the operator $\mathcal G$ is defined as $$(\mathcal G\phi)(t):= \int_0^\infty\int_{\R}K(s,w)g(\phi(t-cs-
w))dwds+f_\beta(\phi(t)).$$ Note that $(\mathcal G\phi)(t)$ can be rewritten as 
 \begin{align*} 
 (\mathcal G\phi)(t)&=\int_\R g(\phi(t-r))\int_0^\infty K(s,r-cs)ds dr+f_\beta(\phi(t))\\
&=\int_\R g(\phi(t-r))k_2(r) dr+f_\beta(\phi(t)),
\end{align*}
where,  by Fubini's Theorem, $$k_2(r)=\int_0^\infty K(s,r-cs)ds,$$
is well defined for all $r\in \R$.
  Consequently, $\phi$ also must satisfy the equation
\begin{align}\label{dk1} 
\nonumber \phi(t)&=(k_1*k_2)*g(\phi)(t)+ k_1* f_\beta(\phi)(t)\\
&=\int_Xd\rho(\tau)\int_\R \mathcal N(s,\tau)g(\phi(t-s),\tau) ds, \,t\in \R, \end{align}
where $X=\{\tau_1,\tau_2\}$, 
\begin{align*}\mathcal N(s,\tau)=\left\{\begin{array}{cc} ( k_1*k_2)(s),& \tau=\tau_1, \\   k_1(s),& \tau=\tau_2, \end{array}\right.\quad  g(s,\tau)=\left\{\begin{array}{cc}  g(s), & \tau=\tau_1, \\ f_\beta(s),& \tau=\tau_2,\end{array}\right. 
\end{align*}
and $*$ denotes convolution $(f*g)(t)=\int_\R f(t-s)g(s)ds$. 

Now we can invoke   the theory developed in \cite{agt} to prove the uniqueness of  positive bounded solution of (\ref{dk1}), vanishing at $-\infty$.  The following lemma shows that the existence of the semi-wavefront with speed $c$ assures that the functions $\chi_0(z,c)$ and $\chi_L(z,c)$ are well defined on $[0,\gamma]$ for some $\gamma>0$.

\begin{lem}\label{ga} Assume that { \bf{$H_1$}} and {\bf{$H_2$}} hold. If $\phi$ is a semi-wavefront  solution of (\ref{i2}) with speed $c$, then there exists $\gamma=\gamma(c)>0$ such that the integrals $$\int_{-\infty}^0\phi(s)e^{-s\gamma}ds \quad \text{ and }\quad  \int_0^\infty \int_{\R}K(s,w)e^{-\gamma(cs+w)}dwds$$ are convergent. 
\end{lem}
\begin{proof}
First, we define the  function $p_\delta(\tau):=\inf_{u\in(0,\delta)}\frac{g(u,\tau)}{u}$ for $\delta>0$, which is a measurable function on $X$. Since $\phi$ satisfies the equation (\ref{dk1}),
we can apply \cite[Theorem 1, p. 77]{agt} to prove that 
$\int_{-\infty}^0\phi(s)e^{-s \gamma}ds $ and  $\int_\R \int_X \mathcal N(s,\tau)p_\delta(\tau)d\rho(\tau)e^{-s\gamma}ds$ are convergent for an appropriate   $\gamma=\gamma(c)>0$. Indeed,  we first  observe that, by the monotone convergence theorem, 
\begin{align*}
&\lim_{\delta \to 0+} \int_\R \int_X \mathcal N(s,\tau)\inf_{u\in(0,\delta)}\frac{g(u,\tau)}{u}d\rho(\tau)ds=\int_\R \int_X \mathcal N(s,\tau)g'(0,\tau)d\rho(\tau)ds \\&=g'(0)\int_\R(k_1*k_2)(s)ds+f'_\beta(0)\int_\R k_1(s)ds
=1+\frac{g'(0)-f'(0)}{\beta}>1.
\end{align*}
Therefore,
\begin{align*}
\int_\R \int_X \mathcal N(s,\tau)p_\delta(\tau)d\rho(\tau)ds\in (1,\infty),
\end{align*}
for all $0<\delta<\delta'$, being $\delta'$ sufficiently small. In this way, since $g(u,\tau) \geq p_{\delta}(\tau)u, \,\,u\in (0,\delta)\subset(0,\delta')$, \cite[Theorem 1, p. 77]{agt} assures that there exists $\gamma=\gamma(c)>0$ such that $$\int_{-\infty}^0\phi(s)e^{-s\gamma}ds \quad \text{ and }\quad \int_\R \int_X \mathcal N(s,\tau)p_\delta(\tau)d\rho(\tau)e^{-s\gamma}ds$$ are convergent. Consequently, since
\begin{align*}
&\int_\R \int_X \mathcal N(s,\tau)p_\delta(\tau)d\rho(\tau)e^{-s\gamma}ds\geq\frac{1}{2}\int_\R \int_X \mathcal N(s,\tau)g'(0,\tau)d\rho(\tau)e^{-s\gamma}ds\geq 0,
\end{align*}
for all $\delta>0$ sufficiently small, we have
\begin{align}\label{23}
\int_\R \int_X \mathcal N(s,\tau)g'(0,\tau)d\rho(\tau)e^{-s\gamma}ds=\frac{g'(0)\int_0^\infty \int_{\R}K(s,w)e^{-\gamma(cs+w)}dwds+f'_\beta(0)}{\beta+c\gamma-\gamma^2}
\end{align}
is finite, and the proof follows.
\end{proof}
Next, let $(\mathcal{L} \mathcal N)(z)$ and $(\mathcal L \phi)(z)$ be the bilateral Laplace transforms   
$$(\mathcal{L} \mathcal N)(z)= \int_\R \int_X \mathcal N(s,\tau)g'(0,\tau)d\rho(\tau)e^{-sz}ds,$$
$$\displaystyle (\mathcal L \phi)(z)=\int_{\R}e^{-z s}\phi(s)ds,\,\,z\in\C.$$
From (\ref{23}) and Lemma \ref{ga} we conclude that $(\mathcal{L} \mathcal N)(z)$  is convergent,  if $0\leq\Re z\leq \gamma$, where $\gamma>0$ is given above. Then we can find  some  maximal number  $\gamma_K(c) \in (0,+\infty]$ such that  $(\mathcal{L} \mathcal N)(z)$ converges, if $\Re z \in [0,\gamma_K(c))$ and diverges,  if $\Re z>\gamma_K(c)$. Similarly, since $\phi$ is positive and bounded,  we  have that  $\mathcal L( \phi)(z)$  is convergent, if $\Re z\in(0,\gamma]$.  Thus we also can get   some maximal number $\gamma_\phi(c)\in (0,+\infty]$,  such that  $\mathcal L( \phi)(z)$ is   convergent, if $0 <\Re z < \gamma_\phi(c)$ and diverges, if $\Re z>\gamma_\phi(c)$. By \cite[Theorem 5b, p. 58)]{widder} $\gamma_K(c)$ and $\gamma_\phi(c)$ are singular point of 
$(\mathcal{L} \mathcal N)(z)$ and $(\mathcal L \phi)(z)$, respectively, if they are finite.

Now, let us analyze separately the integral
\begin{equation}\label{conv} 
\int_0^\infty \int_{\R}K(s,w)e^{-z(cs+w)}dwds.
\end{equation}
\begin{cor} \label{cor1} Assume that { \bf{$H_1$}} and {\bf{$H_2$}} hold. If further there exists  a semi-wavefront  solution of (\ref{i2}) with speed $c$, then  there exists  
an extended  real number $\gamma^\#(c)>0$  such that (\ref{conv}) converges when $z\in [0,\gamma^\#(c))$ and diverges, if $z>\gamma^\#(c)$. Moreover, the function $\gamma^\#(c)$ is increasing on  its domain of definition.
 \end{cor}

\begin{proof}
Suppose that there exists  a semi-wavefront  solution of (\ref{i2}) with speed $c$. It is easy to see that the convergence of (\ref{conv}) for $\gamma>0$ implies its convergence for $z\in [0,\gamma]$. Then   from Lemma \ref{ga} it follows that there exists $\gamma^\#(c)>0$ an extended  number  such that (\ref{conv}) converges when $z\in [0,\gamma^\#(c))$ and diverges, if $z>\gamma^\#(c)$. We now prove  that  $\gamma^\#(c)$ is increasing on its domain of definition. On the contrary, suppose that $c_1<c_2$ and $\gamma^\#(c_1)>\gamma^\#(c_2)$. If  $\gamma$ is such that $\gamma^\#(c_2)<\gamma< \gamma^\#(c_1)$, then we can observe that
\begin{align*}
\int_0^\infty \int_{\R}K(s,w)e^{-\gamma(c_2s+w)}dwds\leq \int_0^\infty \int_{\R}K(s,w)e^{-\gamma(c_1s+w)}dwds<\infty,
\end{align*}
which  contradicts the maximality of $\gamma^\#(c_2)$. \end{proof}

 \begin{lem}\label{34} Suppose that { \bf{$H_1$}} and {\bf{$H_2$}} hold.  Let $\phi$ be  a semi-wavefront  solution of (\ref{i2}) with speed $c$.  Then, without the loss of generality, we have 
$$\gamma_K(c)=\left\{\begin{array}{cc}\gamma^\#(c),& \gamma^\#(c)<\infty \\\mu(c), & \gamma^\#(c)=\infty.\end{array}\right.$$
Moreover, $\gamma_K(c)$  is strictly increasing  on its domain of definition  and $\gamma_\phi(c)\leq \gamma_K(c)$. Finally,  if $\gamma^\#(c)=+\infty$, then $\gamma_\phi(c)< \gamma_K(c)$. \end{lem}
\begin{proof}
 By  \cite[Lemma 1, p. 80]{agt}  we obtain   $(\mathcal{L} \mathcal N)(\gamma_\phi(c))$ is finite, and hence   $\gamma_\phi(c)\leq\gamma_K(c)$. 
Moreover,  from (\ref{23}) is clear that $\gamma_K(c)=\min\{\mu(c), \gamma^\#(c)\}$. If $\gamma^\#(c)<+\infty$, then we can choose a  sufficiently large $\beta$, such that $\mu(c)> \gamma^\#(c)$. Now, if we suppose that $c_1<c_2$, then 
$$\gamma_K(c_1)=\mu(c_1)<\mu(c_2)\leq \gamma_K(c_2),\,\, \text{if} \,\,\gamma^\#(c_1)=+\infty,$$
and 
$$\gamma_K(c_1)=\gamma^\#(c_1)\leq \mu(c_1)<\mu(c_2)\leq \gamma_K(c_2), \,\, \text{if}\,\, \gamma^\#(c_1)<\infty.$$
Which prove that the function $\gamma_K$ is strictly increasing  on its domain of definition.
On the other hand, if $\gamma^\#(c)=+\infty$, then $\gamma_K(c)=\mu(c)$ and 
\begin{align*}
\lim_{z \to \mu(c)-}\frac{g'(0)\int_0^\infty \int_{\R}K(s,w)e^{-z(cs+w)}dwds+f'_\beta(0)}{\beta+cz-z^2}=+\infty.
\end{align*}
 In consequence,   
\begin{align*}
\lim_{z \to \mu(c)-} \int_\R \int_X \mathcal N(s,\tau)g'(0,\tau))d\rho(\tau)e^{-s\gamma}ds=+\infty,\end{align*}
 by (\ref{23}), and finally the inequality 
 $\gamma_\phi(c) < \gamma_K(c)$  follows from \cite[Corollary 1, p. 80]{agt}. 

\end{proof}
Next,  we   establish some properties of $ \mathcal N(s,\tau)$, which will be necessary to apply   \cite{agt} (see conditions {\bf {(EC$_{\gamma_\phi}$)}}, {\bf {(SB)}},  {\bf {(SB*)}} and  {\bf{ (EC*)}}  in Appendix).

\begin{lem} \label{Hip} Assume that {\bf{ $H_1$}} and {\bf{ $H_2$}} hold. Let $\phi$ be  a semi-wavefront  solution of (\ref{i2}) with speed $c$. Then the following statements are valid:
\begin{enumerate}[(i)]
\item There is a measurable function  $C(\tau) >0$  such that 
\begin{equation}\label{ze}
\zeta(z) := \int_\R \int_X C(\tau)\mathcal N(s,\tau)e^{-sz}d\rho(\tau)ds <  +\infty, \,\,\, z \in (0, \gamma_K(c)). 
\end{equation}
\item For $z\in (0, \gamma_K(c))$  there exists a  measurable function 
$ d_{z} \in L^1(X)$ such that  
$$
0\leq \mathcal N(s,\tau) \leq  d_{z}(\tau)e^{z s}, \ s \in \R,\  \tau \in X. 
$$
\item For $z \in (0, \gamma_\phi(c))$  there exists a  measurable function 
$\tilde d \in L^1(X)$ such that  
\begin{equation*}
0\leq \mathcal N(s,\tau) \leq \tilde d(\tau)e^{zs}, \ s \in \R,\  \tau \in X. 
\end{equation*}

\end{enumerate}
\end{lem}
\begin{proof}
Let $C(\tau)$ be given by $C(\tau)=C_1$, if $\tau=\tau_1$ and $C(\tau)=C_2$, if $\tau=\tau_2$, where the constant $C_1,C_2>0$. Then, 
\begin{align*}\label{ze}
\zeta(z) &=\int_\R k_1(s)e^{-sz}ds \left(C_1\int_0^\infty\int_\R K(r,w)e^{-z(w+cr)}dwdr+C_2 \right)\\
&=\frac{1}{\beta+cz-z^2}\left(C_1\int_0^\infty\int_\R K(r,w)e^{-z(w+cr)}dwdr+C_2 \right)
\end{align*}
is convergent for each  $z \in (0, \gamma_K(c))$, by Lemma (\ref{34}).
 We now  show  that there exist  functions $d_j$  on $(0, \gamma_K(c))$ and $\tilde d_j$  on $(0, \gamma_\phi(c))$ such that 
$$
0\leq \mathcal N(s,\tau_j)\leq d_{j}(z)e^{z s}, s \in \R,\, z \in (0, \gamma_K(c)),
$$
$$0\leq \mathcal N(s,\tau_j)\leq \tilde d_{j}(\tau)e^{z s}, s \in \R,\, z \in (0, \gamma_\phi(c)).$$
Indeed, if $j=1$, we have
\begin{align*}
\mathcal N(s,\tau_1) &=   \frac{1}{\sigma(c)}\left[\int_0^\infty\int_{s-cr}^{+\infty}e^{\mu(c) (s-cr-u)}K(r,u)dudr\right.\\
&
\quad+ \left. \int_0^\infty\int^{s-cr}_{-\infty}e^{\nu(c) (s-cr-u)}K(r,u)dudr \right ] \\
&\leq \frac{e^{zs}}{\sigma(c)}\int_0^\infty\int_\R K(r,u)e^{-z (u+cr)}dudr :=d_1(z) e^{zs}, \quad z \in (0, \gamma_K(c)).
\end{align*} 
 In the case $j=2$,   then  we have
\begin{align*}
\mathcal N(s,\tau_2) &=   k_1(s)\leq \frac{e^{zs}}{\sigma(c)}:=d_2 e^{zs},\quad z \in (0, \gamma_K(c)).
\end{align*} 
Finally, for $z \in (0, \gamma_\phi(c))$, we have
\begin{align*}
\mathcal N(s,\tau_1) &\leq  \frac{e^{zs}}{\sigma(c)}\left (\int_0^\infty\int_{-\infty}^{-cr} K(r,u)e^{-\gamma_\phi (u+cr)}dudr+1\right):=\tilde d_1e^{zs}<\infty,
\end{align*} 
by Lemma \ref{34}.
Therefore,
$$\tilde d(\tau)=\left\{\begin{array}{cc} \tilde d_1, & \tau=\tau_1, \\ d_2,& \tau=\tau_2,\end{array}\right., \,\,d_z(\tau)=\left\{\begin{array}{cc}  d_1(z), & \tau=\tau_1, \\ d_2,& \tau=\tau_2,\end{array}\right.$$
 and hence $d_z$ and $\tilde d$ are measurable functions on $(X,\mu)$.  
\end{proof}
 
\section{Characteristic functions.}\label{car}
To guarantee  the existence of $c_*$ and $c_\star$  defined in the Section \ref{into}, we  have  to analyze the real solutions of the equations $\chi_0(z,c)=0$ and $\chi_L(z,c)=0$. Thus it is convenient to consider a  more general equation:  
\begin{equation*}
\mathcal R(z,c):=z^2- cz-q+p\int_0^\infty \int_{\R}K(s,w)e^{-z(cs+w)}dwds=0,
\end{equation*}
where $p>q>0$.

 \begin{lem} \label{c00}
Suppose that given $c\in \R$, the function $\mathcal R(z,c)$ is defined for all $z$ from some maximal  interval $[0,\delta(c))$,   $\delta(c)\in (0,+\infty]$. Then   there exists $ c^\#\in \R$ such that 
\begin{enumerate}[(i)]
\item for any $c> c^\#$, the function $\mathcal R(z,c)$ has  at least  one positive zero  $z=\lambda_1(c) \in (0,\delta(c))$ and  can have  at most two positive zeros  on $(0,\delta(c))$. If the second zero exists, we denote it as $\lambda_2(c)>\lambda_1(c)$. Furthermore, each  $\lambda_j(c)<\mu_q(c)$, where $\mu_q(c)>0$ satisfies the  equation $z^2-cz-q=0$.
\item if $c=c^\#$  and $\lim_{z\uparrow \delta(c^\#)} \mathcal R(z,c^\#)\not=0$, then $\mathcal R(z,c^\#)$  has a unique zero  of  order two on $(0,\delta(c^\#))$, denoted by $z=\lambda_1(c^\#)$, and $\mathcal R(z,c^\#)>0$ for all $z\not=\lambda_1(c^\#)\in [0,\delta(c^\#))$. 
\end{enumerate}
\end{lem}

\begin{proof} 
First note that $\mathcal R(0,c)=p-q>0$ and $\lim _{c\downarrow-\infty} \mathcal R(z,c)=+\infty$ for $z\in (0,\delta(c))$. Since
 \begin{equation*}
 \frac{\partial^2 \mathcal R}{\partial z^2}(z,c)=2+p\int_0^\infty \int_{\R}K(s,w)e^{-z(cs+ w)}(cs+ w)^2dwds>0,\quad z\in [0,\delta(c)),
 \end{equation*}
 the function  $\mathcal R(z,c)$ is strictly convex with respect to $z$, and hence it has  at most two real zeros  for each $c$. Note that if $z=\lambda$ is a zero of $\mathcal R(z,c)$, then 
 $\lambda^2-c\lambda-q<0$, and hence $\lambda<\mu_q(c).$
 On the other hand, for $z\in (0,\delta(c))$ the function $\mathcal R(z,c)$ is  strictly decreasing  in $c$ and $\lim _{c\uparrow+\infty} \mathcal R(z,c)=-\infty$ pointwise,  by  Lebesgue's theorem of dominated convergence. Note here that  $\delta(c)$  is increasing in $c$ and 
    $$\int_0^\infty \int_{\R}K(s,w)e^{-z(cs+ w)}dwds\to 0\,\, \text{as}\,\, c\to+\infty, \,\, \text{for }\,\, z\in (0,\delta(0)).$$
Thus we can define $$c^\#=\inf\{c\in \R: \mathcal R(z,c)<0\,\, \text{for some}\,\,z\in (0,\delta(c)\},$$
and since $\mathcal R(z,c)$ is  strictly decreasing  in $c$, we have $\mathcal R(z,c^\#)\geq 0$ for all $z\in [0,\delta(c^\#))$. It is clear that if $c>c^\#$, then there exists some $z(c)>0$ such that $\mathcal R(z(c),c)<0$. Since $\mathcal R(0,c)>0$, we see that  $\mathcal R(z,c)$ has at least  one zero on $(0,\delta(c))$. By the above argument, $\mathcal R(z,c)$ can have   at most two positive  zeros, and hence we denote by $\lambda_1(c)$ to the minimal root of $\mathcal R(z,c)$ on $(0,\delta(c))$. 

On the other hand, if $\lim_{z\uparrow \delta(c^\#)} \mathcal R(z,c^\#)\not=0$ we assure that there exists a unique $z'\in (0,\delta(c^\#))$ such that 
$$\mathcal R(z',c^\#)=0, \frac{\partial\mathcal R}{\partial z}(z',c^\#)=0,  \,\, \text{and}\,\,\mathcal R(z,c^\#)>0\,\, \text{ for}\,\,z\not=z'\in [0,\delta(c^\#)).$$ Indeed, 
let $\{c_j\}$ be a decreasing sequence $c_j\downarrow c^\#$ such that  $\mathcal R(z_j,c_j)<0$ for some $z_j\in (0,\delta(c_j))$. Since $\mathcal R(0,c_j)>0$ for each $j$, there exists $\lambda(c_j)\in (0,\delta(c_j))$ such that $\mathcal R(\lambda(c_j),c_j)= 0$. We can assume that   $\lambda_1(c_j):=\lambda(c_j)$ is the minimal root of $\mathcal R(z,c_j)$ on $(0,\delta(c_j))$. Note that $\delta(c_j)\downarrow \delta(c^\#)$ and $\lambda_1(c_j)\in (0,\delta(c^\#))$ is strictly increasing when $c_j\downarrow c^\#$.  Thus, there exists some $z'\in (0,\delta(c^\#)]$ such that $\lambda_1(c_j)\uparrow z'$, $\mathcal R(z',c_j)<0$ and hence 
$$\mathcal R(z',c_j)<0\to \mathcal R(z',c^\#)\leq0\,\,\text{ as}\,\,   j\to\infty,$$ by Lebesgue's theorem of dominated convergence. We thus get that $\mathcal R(z',c^\#)=0$ and since $\lim_{z\uparrow \delta(c^\#)} \mathcal R(z,c^\#)\not=0$, we have  $z'\in (0,\delta(c^\#))$. Finally, $\frac{\partial\mathcal R}{\partial z}(z',c^\#)=0$ and $\mathcal R(z,c^\#)>0$ for all $z\not=z'$, because $\mathcal R(z,c^\#)$ is convex with respect to $z$.  We denote $\lambda_1(c^\#)=z'$,
and thus the proof is  complete.
 
 \end{proof}

Based on  \cite{agt} we introduce  the characteristic function $\chi$ associated with the variational equation along the trivial steady state of (\ref{dk1}), by 
$$
\chi(z):= 1-\int_\R \int_X \mathcal N(s,\tau)g'(0,\tau)d\rho(\tau)e^{-sz}ds. 
$$ 
We also will need  the following function 
$$
 \chi_L(z):= 1-\int_\R \int_X \mathcal N(s,\tau)\lambda(\tau)d\rho(\tau)e^{-sz}ds, 
$$ 
where 
\begin{align*}
\lambda(\tau)=\left\{\begin{array}{cc}  L, & \tau=\tau_1, \\ \beta-\inf_{s\geq 0}f'(s),& \tau=\tau_2,\end{array}\right.
\end{align*}
 is measurable function on $(X,\mu)$ with $L\geq g'(0)$.  
 
From now on, we  will say that real number $c$ is an admissible wave speed, if  there exists a semi-wavefront   solution of (\ref{i2}) propagating with velocity  $c$. Note that  $\chi(z)$ is well defined on $[0,\gamma_K(c))$ for each admissible $c$. In the following result  we establish  the relation between the zeros of the functions  $\chi_0(z,c)$, $\chi(z)$, $\chi_L(z,c)$ and  $\chi_L(z)$. Observe that 
\begin{align}\label{p}\nonumber
\chi(z)&=1-g'(0)\int_\R \mathcal N(s,\tau_1)e^{-zs}ds-(\beta-f'(0))\int_\R \mathcal N(s,\tau_2)e^{-zs}ds\\\nonumber&= 1-\frac{\beta-f'(0)}{\beta+cz-z^2}-\frac{g'(0)}{\beta+cz-z^2}\int_0^\infty\int_\R K(r,w)e^{-z(rc+w)}dwdr\\&
=-\frac{\chi_0(z,c)}{\beta+cz-z^2},
\end{align}
and so
\begin{align*}
\chi_L(z)= -\frac{\chi_L(z,c)}{\beta+cz-z^2}.
\end{align*}

 \begin{lem} \label{zero}
Assume that {\bf{$H_1$}} - {\bf{$H_3$}}  hold. Let $\phi$ be  a semi-wavefront  solution of (\ref{i2}) with speed $c'$. Then the following statements are true.
\begin{enumerate}[(i)]
\item The functions    $\chi_0(z,c')$  and $\chi_L(z,c')$ are well defined on $[0,\gamma_K(c'))$.
\item  The  equation $\chi_0(z,c')=0$ has at least one root $\lambda_1(c')\in (0,\gamma_\phi(c')]\subset (0,\gamma_K(c')]$. 
\item If  further we assume that  $g(s)\leq Ls$, $f_\beta(s)\leq (\beta-\inf_{s\geq 0}f'(s))s$, $s\geq 0$,   and  if  there exists   $m\in (0,\gamma_K(c'))$ such that $\chi_L(m,c')\leq 0$, then   $\lambda_1(c)=\gamma_\phi(c)\leq m<\gamma_K(c)$ for each admissible wave speed $c\geq c'$.
\end{enumerate}
  \end{lem}
  \begin{proof}
  Since  $c'$ is an admissible wave speed, Lemma \ref{34} and  (\ref{p}) imply  that   $\chi(z)$ and $\chi_0(z,c')$ are well defined on $[0,\gamma_K(c'))$. Note that  $\chi_0(z,c')$ and $\chi_L(z,c')$ have the same interval of convergence. Hence $\chi_L(z,c')<\infty$, if  $z\in [0,\gamma_K(c'))$. Moreover,
$$\chi(0)=-\frac{\chi_0(0,c')}{\beta}=\frac{f'(0)-g'(0)}{\beta}<0.$$
  From \cite[Theorem 2, p. 81]{agt} we get that   $\chi(z)$   has a zero on $(0, \gamma_\phi(c')]\subset (0, \gamma_K(c')]$, and from   (\ref{p}), we see that $\chi_0(z,c')$  also has a zero  $z'\in (0, \gamma_\phi(c')]$. Note that, by  Lemma \ref{c00}, $\chi_0(z,c')$, and hence  $\chi(z)$ can have at most two positive zeros on $(0,\gamma_K(c'))$.

On the other hand, 
\begin{align}\label{L}
\chi_L(z)= -\frac{\chi_L(z,c')}{\beta+c'z-z^2}\leq-\frac{\chi_0(z,c')}{\beta+c'z-z^2}=\chi(z), \quad z\in [0,\gamma_K(c')).
\end{align}
From  (\ref{L}) and the condition $\chi_L(m,c')\leq 0$ with  $m\in (0,\gamma_K(c'))$, it follows that  $\chi_L(z)$ is well defined on $[0,\gamma_K(c'))$ and $\chi_L(m)\geq 0$. In addition, since  $g$ and $f$  satisfy 
\begin{align*}
g(s)\leq Ls,\,\,f_\beta(s)\leq (\beta-\inf_{s\geq 0}f'(s))s,\,s>0,
\end{align*}
 we obtain  $g(s,\tau)\leq \lambda(\tau)s, \,\, s>0$. Therefore  \cite[Lemma 6, p.88]{agt} implies that $\gamma_\phi(c')$  coincides with the minimal positive zero of $\chi(z)$, and hence $z'=\gamma_\phi(c')$. We denote $\lambda_1(c')=\gamma_\phi(c')$. In addition, since $\chi_0(m,c')\leq \chi_L(m,c')\leq 0$, we have  $$\lambda_1(c')=\gamma_\phi(c')\leq m<\gamma_K(c').$$ In this way, observe that  $\chi_L(m,c)$ is decreasing in $c$, and hence   $\chi_L(m,c)< \chi_L(m,c')\leq 0$ for  $c>c'$. Similarly to above,  \cite[Lemma 6]{agt} also allows to prove that    $\lambda_1(c)=\gamma_\phi(c)$ for each admissible wave speed $c>c'$. Finally, since  the functions $\lambda_1(c)$ is strictly  decreasing  and $\gamma_K(c)$ is strictly increasing  in $c$, we have that
 \begin{align*}
 \lambda_1(c)=\gamma_\phi(c)<\lambda_1(c')=\gamma_\phi(c')\leq m< \gamma_K(c')<\gamma_K(c),
 \end{align*} for each admissible wave speed  $c>c'$.  This completes the proof of lemma.

 \end{proof}
\begin{lem} Suppose that { \bf{$H_1$}} and {\bf{$H_2$}} hold.  If $c\in \R$ is an admissible wave speed, then
\begin{align}\label{ction2}
c>-\frac{g'(0)\int_0^\infty\int_\R K(s,w)w\,dwds}{1+g'(0)\int_0^\infty\int_\R K(s,w)s\,dwds}.
\end{align}
Hence, the estimation (\ref{speed}) is valid for each admissible wave speed, if
$$
\int_0^\infty\int_\R K(s,w)w\,dwds\leq 0.
$$

\end{lem}
\begin{proof}   Let $c\in \R$ be an admissible wave speed. Then Lemma \ref{zero} implies that $\chi_0(z,c)$ is well defined on $[0,\gamma_K(c))$ and has at least one root $(0,\gamma_K(c)]$. Thus   $\frac{d}{dz}(\chi_0(0,c))<0$, and therefore 
$$c\Big(1+g'(0)\int_0^\infty\int_\R K(s,w)s\,dwds\Big)>-g'(0)\int_0^\infty\int_\R K(s,w)w\,dwds,$$
which  gives  (\ref{ction2}).   Thus  the proof complete.
\end{proof}

 \section {Non-existence and uniqueness of positive  semi-wavefront.}\label{main}
 
 In this section, we first  prove the non-existence result  given by Theorem \ref{3}.  Next, we   study the uniqueness of semi-wavefront developing a version which is  more complete than   Theorem \ref{2}  announced in the introduction. 
  \begin{proof}[Proof. Theorem \ref{3}]   
First, note that Lemma \ref{c00} guarantees the existence of $c_*$ as the minimal value of $c$  for which the equation $\chi_0(z,c)=0$ has at least one positive root. If we suppose that for $ c<c_*$ there exists  a semi-wavefront  solution of (\ref{i1}) with speed $c$,  then Lemma \ref{zero} implies that  $\chi_0(z,c)$ is well defined on $[0,\gamma_K(c))$ and   $\chi_0(z',c)=0$ for some $z'\in (0,\gamma_K(c)]$, which contradicts the minimality of $c_*$. 
\end{proof} 

\begin{lem}\label{beta} Suppose that condition  {\bf{ $H_2$}} holds. If $M$ is a positive constant,  there exists $\beta=\beta(M)>0$ sufficiently large such that $f_\beta(s)\geq 0$ for all $s\geq 0$ and
\vspace{-0,1cm}
\begin{align*}
|f_\beta(s_1)-f_\beta(s_2)|\leq \Big(\beta-\inf_{s\geq 0}f'(s)\Big)|s_1-s_2|,\,\,s_1,s_2\in[0,M].
\end{align*}
\end{lem}
\begin{proof} Let $M$ be any positive number. Since $f$ is  continuously differentiable on $[0,M]$ and $f(0)=0$,  we can choose $\beta>\inf_{s\geq 0}f'(s)$ sufficiently large   such that $f_\beta(s)=\beta s- f(s)\geq 0$ for all $s\in [0,M]$ and 
\vspace{-0,2cm}
$$\max _{s\in[0,M]}f'(s)\leq 2\beta-\inf_{s\geq 0}f'(s).$$
By the Mean Value Theorem, it follows that   $f(s_2)-f(s_1)=f'(s_0)(s_2-s_1)$ for some $s_0\in [s_1,s_2]\subset [0,M]$. Thus we get
\vspace{-0,1cm}
\begin{align}\label{ff1}
\frac{f_\beta(s_2)-f_\beta(s_1)}{s_2-s_1}= \beta -\frac{f(s_2)-f(s_1)}{s_2-s_1}= \beta-f'(s_0)\leq \beta-\inf_{s\geq 0}f'(s),
\end{align}
\vspace{-0,1cm}
and 
\vspace{-0,2cm}
\begin{align}\label{ff2}
\frac{f_\beta(s_2)-f_\beta(s_1)}{s_2-s_1}\geq \beta-\Big(2\beta-\inf_{s\geq 0}f'(s)\Big)=-\beta+\inf_{s\geq 0}f'(s).
\end{align}
Finally, we conclude from (\ref{ff1}) and (\ref{ff2})  that
\begin{align*}
\left|f_\beta(s_2)-f_\beta(s_1)\right|\leq  \Big(\beta-\inf_{s\geq 0}f'(s)\Big)|s_2-s_1|,\quad s_1,s_2\in [0,M].
\end{align*}
\end{proof}
\vspace{-0.5cm}
\begin{remark}\label{not}
In order to get the uniqueness result, it is necessary  to assume that $f_\beta$  is a Lipschitzian function such that 
$$|f_\beta(s_1)-f_\beta(s_2)|\leq \Big(\beta-\inf_{s\geq 0}f'(s)\Big)|s_1-s_2|,\,\,s_1,s_2\geq 0.$$
We note that there is no loss of generality in assuming this condition because the proof of the uniqueness in \cite{agt}  compares two solutions $\phi_1$ and $\phi_2$, which are uniformly bounded on $\R$ by  $M:=\max\{\sup_{t\in \R}\phi_1(t),\sup_{t\in \R}\phi_2(t)\}$, and only  involves the values of $f_\beta(\phi_j(s))$.
\end{remark}

 \begin{thm} \label{mth1}  Assume  {\bf{$H_1$}} - {\bf{$H_3$}}   and suppose that $g$ satisfies the condition (\ref{lip}). Let $\phi$ be  a semi-wavefront  solution of  (\ref{i2}) with speed $c'$. If there exists $m\in (0,\gamma_K(c'))$ such that 
$\chi_L(m,c')\leq 0$, then   $\phi$ is  the  unique semi-wavefront  solution of  (\ref{i2}) (modulo translation).  Moreover,  the uniqueness also holds for each semi-wavefront solution with speed $c> c'$. 
  \end{thm}

\begin{proof}
The  proof will be divided into 3  steps.\\
\underline{Step I}. 
Our proof starts  by observing that    the Lipschitz condition (\ref{lip}) and Remark \ref{not} allow to assume that
$$|g(s_1,\tau)-g(s_2,\tau)|\leq \bar\lambda(\tau)|s_1-s_2|,\,\, s_1,s_2\geq 0,\,\,\tau\in X,$$
where  $\bar\lambda(\tau)=\lambda(\tau)$, if $ f'(0)\not=\inf_{s\geq 0}f'(s)$ and  $\bar\lambda(\tau)=g'(0,\tau)$, otherwise.

On the other hand, since $f$ and $g$ satisfy condition {\bf{ $H_3$}} and $g(0)=f(0)=0$, there exist  appropriate $ C_1,C_2,\sigma>0$  such that 
$$|g(u)-g'(0)u|\leq C_1u^{\alpha+1},\quad | f_\beta(u)-f_\beta'(0)u|=| f(u)-f'(0)u|\leq C_2u^{\alpha+1},\,\,u\in (0,\sigma),$$
and hence $g(s,\tau)$ satisfies  
$$|g(u,\tau)- g'(0,\tau)u| \leq C(\tau)u^{1+\alpha}, \ u \in (0,\sigma),$$ where the function $C(\tau)$ is constant on $X$. From this and Lemma \ref{Hip}  we see that the assumptions  {\bf {(SB*)}}, {\bf{ (EC*)}}, {\bf {(EC$_{\gamma_\phi}$)}} and  {\bf {(SB)}} (except $\gamma_\phi(c) < \gamma_K(c)$) of  \cite{agt} hold. \\
\underline{Step II}. Now suppose that $f'(0)\not=\inf_{s\geq 0}f'(s)$. Then $\chi_0(z,c')<\chi_L(z,c')$ and since  $\chi_L(m,c')\leq 0$ for some $m\in (0,\gamma_K(c'))$, we have $m<\lambda_2(c')$, if $\lambda_2(c')$ exists. Thus the function $\chi_L(m)\geq 0$ for $m\in (0,\lambda_2(c'))$, and since 
$\chi(0)<0$,  \cite[Theorem 4, p. 96]{agt} (see Theorem \ref{agt2} in Appendix)  implies  that  $\phi$ is unique (modulo translation).  In addition,  we also obtain  the uniqueness  of semi-wavefront solution of  (\ref{i2})  for each admissible wave speed $c>c'$, because $\chi_L(m,c)<\chi_L(m,c')$ for $c>c'$ and if $\lambda_2(c)$ exists, then $m<\lambda_2(c)$.\\
\underline{Step III}. In the case, $f'(0)=\inf_{s\geq 0}f'(s)$ and $L=g'(0)$,  we have  $\chi_L(z,c')=\chi_0(z,c')$ and $\chi_0(m,c')\leq 0$.  Note further that   $g(s,\tau)\leq \lambda(\tau)s, s\geq 0$,  and hence  $$\lambda_1(c)=\gamma_\phi(c) \leq m<\gamma_K(c)$$ for each admissible wave speed $c\geq c'$, by Lemma \ref{zero}. Consequently, \cite[Theorem 3, p. 91]{agt} (see Theorem \ref{agt1} in Appendix)  implies the uniqueness (modulo translation) of   semi-wavefront solution of  (\ref{i2}) with speed $c\geq c'$. This completes the proof.
\end{proof}

 \begin{proof}[Proof of Theorem \ref{2}]
  Note that   Lemma \ref{c00}  guarantees the existence of the minimal numbers  $c_\star$ for which  $\chi_L(z,c)$ has at least one positive zero $z=\gamma_1(c)\in(0,\gamma^\#(c))$ for all $c> c_\star$ and for $c\geq c_\star$, if $ \chi_L(\gamma^\#(c_\star)-,c_\star)\not=0$. When $f'(0)=\inf_{s\geq 0}f'(s)$ and $L=g'(0)$, we have  $c_\star=c_*$,  otherwise  $c_\star>c_*$. We observe that  if $c$ is an admissible wave speed, then $c\geq c_*$, by Theorem \ref{3}. 
  
  Next,  let  $c\geq c_\star$ be an admissible wave speed. If $\gamma^\#(c_\star)=+\infty$, then we  have  $\gamma^\#(c)=+\infty$,  and  hence $\gamma_K(c)= \mu(c)$.  Since $\gamma_1(c)<\mu(c)$, by Lemma \ref{c00}, it follows that  $\gamma_1(c)\in (0,\gamma_K(c))$.   In the case  $\gamma^\#(c_\star)<+\infty$,  we see that  either $\gamma_K(c)=\mu(c)$ or $\gamma_K(c)=\gamma^\#(c)$. In both cases we conclude that  $\gamma_1(c)<\gamma_K(c)$ for each admissible wave speed $c\geq c_\star$, if $\chi_L(\gamma^\#(c_\star)-,c_\star)\not= 0$, and  for each admissible $c> c_\star$, if $\chi_L(\gamma^\#(c_\star)-,c_\star)= 0$. Thus Theorem \ref{mth1} implies the uniqueness (modulo translation) the semi-wavefront solution  to  (\ref{i1}) with speed $c\geq c_\star$, if $\chi_L(\gamma^\#(c_\star)-,c_\star)\not= 0$, and with speed  $c>c_\star$,  otherwise.
   \end{proof}
 
 \section {Applications.}
 In this section, we apply Theorem \ref{2}  to some  non-local  reaction-diffusion  epidemic and population models  with distributed time  delay, studied in \cite{AF,FWZ,gouk,OG,tz,wlr2,wl2,wl3,xz}.   

  \vspace{0,3cm}
 {\bf{ An application to  the epidemic dynamics:}} Consider the following reaction-diffusion model with distributed delay
\begin{equation}\label{sy4}
\left\{\begin{array}{c}u_t(t,x) =  du_{xx}(t,x)- f(u(t,x)) + \int_{\R}K(x-y)v(t,y)dy \\\\
v_t(t,x) =  - \alpha v(t,x)+\int_0^\infty g(u(t-s,x))P(ds),  \quad\quad\quad\quad\end{array}\right.
\end{equation}
where $\alpha, d>0$, $x\in \R, t\geq 0$, and $P$ is a probability measure on $\R_+$. The functions $u(t,x)$ and $v(t,x)$ denote the  densities of the infectious agent and the infective human population at a point $x$ in the habitat at time $t$, respectively (see \cite{tz,wl2,wl3,xz}). Note that system (\ref{sy4}) can be seen as a generalization of the systems studied  in the cited works. However, here the nonnegative  kernel $K$ can be asymmetric and normalized by $\int_\R K(w)dw=1$,  and the function $g$ can be non-monotone. By scaling the variables, we can suppose that $d=1$.

Now, suppose that $(u(t,x),v(t,x)) =(\phi(x+ct),\psi(x+ct))$ is a semi-wavefront solution of system (\ref{sy4}) with speed $c$, i.e. the continuous non-constant  uniformly bounded functions $u(t,x)=\phi(x+ct)$  and $v(t,x)=\psi(x+ct)$ are positives  and  satisfy the condition $\phi(-\infty)=\psi(-\infty)=0$.  Then  the wave profiles $\phi$ and $\psi$ must satisfy the following system:
\begin{equation}\label{sy5}
\left\{\begin{array}{c}
\phi''(t)-c\phi'(t)- f(\phi(t)) + \int_{\R}K(u)\psi(t-u)du=0 \\\\
c\psi'(t) +\alpha \psi(t)-\int_0^\infty g(\phi(t-cs))P(ds)=0.  \quad\quad\end{array}\right.
\end{equation}
Integrating the second equation of system (\ref{sy5}) between $-\infty$ and $t$, we find  that $\psi$ satisfies 
\begin{align*}
\psi(t)&=\frac{1}{c}\int_0^\infty\int_0^\infty e^{-\frac{\alpha}{c} u}g(\phi(t-u-cr)) P(dr)du\\
&= \int_0^\infty\int_r^\infty e^{-\alpha(w-r)}g(\phi(t-cw)) dw P(dr)\\
&=\int_0^\infty\int_0^w e^{-\alpha(w-r)}g(\phi(t-cw))  P(dr)dw\\
&=\int_0^\infty g(\phi(t-cw))K_2(w) dw,\,\, c\not=0,
\end{align*} 
where
$$K_2(w)=\int_0^w e^{-\alpha(w-r)} P(dr).$$
Note that if $c=0$, then $\alpha \psi(t)=g(\phi(t))$.
Now, if we rewrite the first equation of system (\ref{sy5}) as  (\ref{ii3}), then  $\phi(t)$  should satisfy the  integral equation
\begin{align*}  
\phi(t)&=\frac{1}{\sigma(c)}\left(
\int_{-\infty}^te^{\nu(c)(t-s)}(\mathcal G\phi)(s)ds
+\int_t^{+\infty}e^{\mu(c)(t-s)}(\mathcal G\phi)(s)ds\right)\\\nonumber
&= \int_\R k_1(t-s)(\mathcal G\phi)(s)ds,
\end{align*}
where $$k_1(s)=(\sigma(c))^{-1}\left\{\begin{array}{cc}e^{\nu(c) s}, & s\geq 0 \\e^{\mu(c) s}, & s<0\end{array}\right.,$$ $\sigma(c)=\sqrt{c^2+4\beta}$,  $\nu(c)<0<\mu(c)$ are the roots of  $z^2-cz-\beta=0$
 and  the operator $\mathcal G$ is defined as $$(\mathcal G\phi)(t):= \int_{\R}K(u)\psi(t-u)du+f_\beta(\phi(t)),\quad f_\beta(s)=\beta s-f(s), \beta>f'(0).$$ 
 In consequence,  
 \begin{align*} 
\phi(t)&= \int_\R k_1(t-s) \left(\int_{\R}K(u)\psi(s-u)du+f_\beta(\phi(s))\right)ds\\
&=\int_\R k_1(t-s) \left(\frac{1}{\alpha}\int_0^\infty\int_{\R} \bar K(w,u) g(\phi(s-cw-u))dudw+f_\beta(\phi(s))\right)ds,
\end{align*}
 where $\bar K(w,u)=\alpha K(u)K_2(w)$, $u\in \R, w\in [0,\infty)$ and $c\not=0$. Since 
 \begin{align*} 
\frac{1}{\alpha}\int_0^\infty\int_{\R} \bar K(w,u) g(\phi(t-cw-u))dudw&=\frac{1}{\alpha}\int_\R g(\phi(t-r))\int_0^\infty \bar K(s,r-cs)ds dr\\
&=\int_\R g(\phi(t-r))k_2(r) dr,
\end{align*}
where $k_2(r)=\frac{1}{\alpha}\int_0^\infty \bar K(s,r-cs)ds$, the profile $\phi$ also must satisfy the equation
\begin{align*}
\nonumber \phi(t)&=(k_1*k_2)*g(\phi)(t)+ k_1* f_\beta(\phi)(t).
 \end{align*}
A similar argument can be applied when $c=0$.
 
 Next, observe here that the characteristic function $\chi$ becomes: 
 \begin{align}\label{ppp}\nonumber
\chi(z)&=1-g'(0)\int_\R (k_1*k_2)(s)e^{-zs}ds-(\beta-f'(0))\int_\R k_1(s)e^{-zs}ds\\\nonumber&= 1-\frac{\beta-f'(0)}{\beta+cz-z^2}-\frac{g'(0)}{\beta+cz-z^2}\frac{\int_0^\infty e^{-zcr}P(dr)}{cz+\alpha}\int_\R K(w)e^{-zw}dw\\&
=-\frac{z^2-cz-f'(0)+\frac{g'(0)}{cz+\alpha}\int_0^\infty e^{-zcr}P(dr)\int_\R K(w)e^{-zw}dw}{\beta+cz-z^2}
\end{align}
when  $cz+\alpha>0$. Consequently, from (\ref{ppp}) we obtain
\begin{equation*}\label{c1}
\chi_0(z,c)=z^2-cz-f'(0)+\frac{g'(0)}{cz+\alpha}\int_0^\infty e^{-zcr}P(dr)\int_\R K(w)e^{-zw}dw.
\end{equation*}
Similarly, we get  that 
\begin{equation*}\label{c2}
\chi_L(z,c)=z^2-cz-\inf_{s\geq 0}f'(s)+\frac{L}{cz+\alpha}\int_0^\infty e^{-zcr}P(dr)\int_\R K(w)e^{-zw}dw. 
\end{equation*}
In this way, let $c_*$ and $c_\star$ be the minimal value of $c$ for which $\chi_0(z,c)=0$ and $\chi_L(z,c)=0$ have at least one positive root,  respectively. Then  we can now formulate the following result:
\begin{thm}\label{apli1}
Let  assumptions {\bf{$H_1$}} - {\bf{$H_3$}} hold.  Suppose further that for any $c\in \R$, there exists some $\gamma^\#=\gamma^\#(c)\in (0,+\infty]$ such that  $\chi_0(z,c)<\infty$ for each $z \in [0,\gamma^\#)$ and  $ \chi_0(\gamma^\#(c)-,c)=+\infty$. If $g$ satisfies the condition (\ref{lip}), then the 
 system (\ref{sy4}) admits at most one (modulo translation) semi-wavefront solution $$(u(t,x),v(t,x))=(\phi(x+ct),\psi(x+ct)), \,\, \phi(-\infty)=\psi(-\infty)=0,$$   for each admissible wave speed $c\geq c_\star$. 
Furthermore, the system (\ref{sy4}) has no any semi-wavefront solution propagating with speed $c<c_*$. \end{thm}
\begin{proof} Note that the condition  $ \chi_0(\gamma^\#(c_\star)-,c_\star)=+\infty$ implies that $ \chi_L(\gamma^\#(c_\star)-,c_\star)\not=0$. Thus 
the proof is a direct consequence of  the Theorems (\ref{2}) and (\ref{3}). 
\end{proof}
\begin{remark}
Theorem \ref{apli1} complement or  improve  some results  of \cite{tz,wl2,wl3,xz}. In fact, in these references the monotone case was studied, except \cite{wl3}. Moreover,  in the mentioned papers, the Lipschitz condition (\ref{lip}) with $L=g'(0)$  was  assumed and the uniqueness of the slowest semi-wavefront was not studied.  It should be noted that in \cite{wl2,xz}, isotropic kernels were  considered.
\end{remark}

 \vspace{0,3cm}
{\bf{An application to the population dinamics:}}  Let $u$ and $v$ denote the numbers of mature and immature population of a single species at time $t\geq 0$, respectively. Then Aiello and Freedman \cite{AF} proposed that the population growth can be modeled by the following system:
\begin{equation}\label{sy}
\left\{\begin{array}{c}u'(t) =  \alpha e^{-\gamma \tau} u(t-\tau)- \beta u^2(t) \quad \quad \quad\\\
v'(t) =  \alpha u(t)-\gamma v(t)- \alpha e^{-\gamma \tau} u(t-\tau),\end{array}\right.
\end{equation}
 where $\alpha,\beta,\gamma, \tau>0$. The delay $\tau$ is the time taken from birth to maturity. Death of immature and  mature are  modeled, respectively, by the  $- \gamma v(t,x)$ and $- \beta u^2(t)$ terms. The $\alpha u(t)$ term denotes the rate at which individuals are born. The term $\alpha e^{-\gamma \tau} u(t-\tau)$ represents the rate  at which individuals leave the immature and enter the mature class. When the  individuals are allowed to move around, Gourlley and Kuang \cite{gouk} introduced a diffusive term to the model (\ref{sy}). To improve population model of \cite{AF,gouk}, Olmari and Gourley \cite{OG} proposed the following nonlocal reaction-diffusion system with distributed time delay:
 \begin{align}\label{sy1}
\left\{\begin{array}{c}u_t(t,x) =  du_{xx}(t,x)- \beta u^2(t,x)) +\alpha \int_0^\infty\int_{\R}K(s,y)u(t-s,y)e^{-\gamma s}f(s)dyds \quad \quad\quad\quad  
\\\\
v_t(t,x) =  Dv_{xx}(t,x)- \gamma v(t,x)+\alpha u(t,x)-\alpha \int_0^\infty\int_{\R}K(s,y)u(t-s,y)e^{-\gamma s}f(s)dyds,\end{array}\right.
\end{align}
where $$K(s,y)=\frac{1}{\sqrt{4\pi Ds}}e^{\frac{-(x-y)^2}{4Ds}},$$
and $u(t,x)$ and $v(t,x)$ denote the density of the mature and immature population of a single species at time $t\geq 0$ and location $x$, respectively. Fang et al. \cite{FWZ} proposed a generalization for (\ref{sy1}) with a general isotropic kernel $K$.

Following \cite{FWZ}, here will study the system 
\begin{equation}\label{sy3}
\left\{\begin{array}{c}u_t(t,x) =  du_{xx}(t,x)- f(u(t,x)) + \int_0^\infty\int_{\R}K(s,w)g(u(t-s,x-w))dwds\quad\quad\quad\quad \\\\
v_t(t,x) =  Dv_{xx}(t,x)- \gamma v(t,x)+g(u(t,x)) -\int_0^\infty\int_{\R}K(s,w)g(u(t-s,x-w))dwds,\end{array}\right.
\end{equation}
which generalize (\ref{sy1}). Here $\gamma, D,d>0$  and  the nonnegative kernel $K$ can be asymmetric. Note that by scaling the variables, we can suppose that $d=1$. When $g(t)=\alpha t$, $K$ satisfies $K(s,w)=K(s,-w)$ and $ \int_0^\infty\int_{\R}K(s,w)<1$, the existence and nonexistence of traveling waves solution of  the
 system (\ref{sy3}) was proved in \cite{FWZ}.  
 Now, observe that in the system (\ref{sy3}) the first equation can be solved independently of the second. In this way,  if the 
 system (\ref{sy3}) admits   a semi-wavefront solution $$(u(t,x),v(t,x))=(\phi(x+ct),\psi(x+ct)),\,\,\phi(-\infty)=\psi(-\infty)=0,$$  with speed $c$, then $v(t,x)=\psi(x+ct)$ must satisfy the immature equation 
 \begin{align*}
D\psi''(t)-c\psi'(t)- \gamma \psi(t)+(\mathcal H\phi)(t)=0,
\end{align*}
where the operator $\mathcal H$ is defined by
 $$(\mathcal H\phi)(t)=g(\phi(t))-\int_0^\infty\int_{\R}K(s,w)g(\phi(t-cs-w))dwds.$$ Since $ \phi$ is  bounded, we get that $\psi$ is represented  by 
\begin{align*}
\psi(t)=\int_\R k_1(t-s)(\mathcal H\phi)(s)ds=\int_\R k_1(s)(\mathcal H\phi)(t-s)ds,
\end{align*}
where $$k_1(s)=\left(\sqrt{c^2+4D\gamma}\right)^{-1}\left\{\begin{array}{cc}e^{\tilde\nu(c) s}, & s\geq 0 \\e^{\tilde\mu(c) s}, & s<0\end{array}\right.$$
 and $\tilde\nu(c)<0<\tilde\mu(c)$ are the roots of  $Dz^2-cz-\gamma=0$. 

Finally, consider  the characteristic functions $\chi_0(z,c)$ and $\chi_L(z,c)$  associated with the mature equation of system (\ref{sy3})
 and  $c_*,c_\star$  defined in Section \ref{into}. Then the following theorem  is a direct consequence of  the Theorems (\ref{2}) and (\ref{3}). 

\begin{thm}\label{un}
Let  assumptions   {\bf{$H_1$}} - {\bf{$H_3$}} hold.  Suppose further that for any $c\in \R$, there exists some $\gamma^\#=\gamma^\#(c)\in (0,+\infty]$ such that  $\chi_0(z,c)<\infty$ for each $z \in [0,\gamma^\#)$ and  $ \chi_0(\gamma^\#(c)-,c)=+\infty$. If $g$ satisfies the condition (\ref{lip}), then the 
 system (\ref{sy3}) admits   at most one   (modulo translation) semi-wavefront solution $$(u(t,x),v(t,x))=(\phi(x+ct),\psi(x+ct)),\,\,\phi(-\infty)=\psi(-\infty)=0,$$  for each admissible wave speed $c\geq c_\star$. 
 Furthermore,  the system (\ref{sy3}) has no any semi-wavefront solution propagating with speed $c<c_*$. 
\end{thm}

\begin{remark}
We note that Theorem \ref{un}  complements or improves some results  of  \cite{FWZ, gouk,tz,wlr2}, where the non-existence or the uniqueness  was established under assumptions  that $K$ is Gaussian or symmetric kernel,  and $g$ monotone.  In \cite{gouk,wlr2} only  the particular cases $f(s)=\beta s^2$ and $g(s)=s$, were studied, and in \cite{tz}, the assumptions were either $f(s)=f'(0)$ or $g(s)=g'(0)s$. Neither of this references considered the uniqueness of the minimal wave (see also \cite{OG}).
\end{remark}
\section{Appendix.} The following assumptions on $g$ and $\mathcal N$ were used in  \cite{agt}.
\begin{description}
\item[ {\bf{ (SB)}}] $\gamma_\phi < \gamma_K$ and, for some measurable  $C(\tau) >0$ and $\alpha, \sigma \in (0,1]$,
$$
|g'(0,\tau)-\frac{g(u,\tau)}{u}| \leq C(\tau)u^{\alpha}, \ u \in (0,\sigma), $$
measurable $C(\tau) >0$ satisfying (\ref{ze}).
\end{description}
\begin{description}
\item[ {\bf{ (EC$_\rho$)}}] For any $\rho < \gamma_\phi$ and  there exist  measurable 
$d_1, d_2, \  d_1d_2 \in L^1(X), $ such that  
$$
0\leq \mathcal  N(s,\tau) \leq d_1(\tau)e^{\rho s}, \ s \in \R,\  \tau \in X,  
$$
\begin{equation*}
|g(u,\tau)| \leq d_2(\tau)u, \ u \geq 0. 
\end{equation*}
\end{description}

\begin{description}
\item[ {\bf{ (SB*)}}] For some $\alpha, \sigma \in (0,1]$ and measurable $C(\tau) >0$ satisfying (\ref{ze}),
$$
|g'(u,\tau)- g'(0,\tau)| \leq C(\tau)u^{\alpha}, \ u \in (0,\sigma), 
$$
it holds. Furthermore, there exist  $\hat \epsilon \in (0,\gamma_\phi)$ and measurable 
$d_1(\tau)$ such that  
$$
0\leq \mathcal  N(s,\tau) \leq d_1(\tau)e^{\hat \epsilon s}, \ s \in \R. 
$$
\end{description}
\begin{description}
\item[ {\bf{ (EC*)}}] 
  There exists  $\delta_0 >0$ such that, for each $x \in (\lambda_{rK}-\delta_0, \lambda_{rK})$, it holds 
$$
0\leq \mathcal  N(s,\tau) \leq d_{2x}(\tau)e^{x s}, \ s \in \R, 
$$
for some $\mu-$measurable 
$d_{2x}(\tau)$.   
\end{description}

\vspace{0.1 cm}
The main results obtained in \cite{agt} are the following:
\begin{thm} \label{agt1}
Assume  {\bf {(SB)}}   as well as  {\bf {(EC$_{\gamma_\phi}$)}}    and suppose further that  $\chi(0)<0$,
\begin{equation*}
|g(u,\tau) - g(v,\tau)| \leq g'(0,\tau)|u-v|, \ u,v \geq 0.
\end{equation*}
Then equation (\ref{dk1}) has at most one bounded positive  solution $\varphi, \  \varphi(-\infty) =0$. 
\end{thm}

\vspace{0.1 cm}
\begin{thm}\label{agt2}
 Assume {\bf {(SB*)}}, {\bf{ (EC*)}} and suppose that
\begin{equation*}
|g(u,\tau) - g(v,\tau)| \leq \lambda(\tau)|u-v|, \ u,v \geq 0, \tau \in X, 
\end{equation*}
for some measurable $\lambda(\tau)$ different from $g'(0,\tau)$  and that the
function 
$$
\tilde\chi(z) =1 - \int_\R \int_X \mathcal N(s,\tau)\lambda(\tau)d\rho(\tau)e^{-sz}ds 
$$
is well defined on $[0, \lambda_{2K})$. If, in addition,  $\lambda d_j \in L^1(X), \  j =1,2,$  $\chi(0) <0$ and
$\tilde\chi(m) \geq 0
$ for some $m \in (0, \lambda_{2K})$, 
then equation (\ref{dk1}) has at most one bounded positive  solution $\varphi, \  \varphi(-\infty) =0$. Here, $\lambda_{2K}$ is the second positive zero of $\chi(z)$, if exists.
\end{thm}

\section*{Acknowledgements} The author thanks Professor Sergei Trofimchuk for valuable discussions and helpful comments. This work was supported by FONDECYT/INICIACION/ Project  11121457.


\begin{thebibliography}{10}

\bibitem{agt} {M. Aguerrea, C. Gomez, S.  Trofimchuk,} {On uniqueness of semiwavefronts. Diekmann-Kaper theory of a nonlinear convolution equation,} {Mathematische Annalen.} {354} (2012) 73-109.

\bibitem{map} {M. Aguerrea,} {Existence of fast positive wavefronts for a non-local delayed reaction�diffusion equation,} {Nonlinear Analysis.} {72} (2010) 2753-2766.

\bibitem{atv}{M. Aguerrea,  S. Trofimchuk, G. Valenzuela,} {Uniqueness of fast travelling fronts in reaction-diffusion equations with delay,}
{Proc. R. Soc. A.} {464} (2008) 2591-2608.

\bibitem{av}{M. Aguerrea and  G. Valenzuela,} {On the first critical speed of travelling waves for non-local delay reaction-diffusion equation,} 
{Nonlinear Oscillations.} {13} (2010)  3-8.

\bibitem{AF}{W.G. Aiello, H.I. Freedman, A time-delay model of single species growth with stage structure,} {Math. Biosci.} {101} (1990) 139-153.

\bibitem{Ai}{S. Ai,} {Traveling wave fronts for generalized Fisher
equations with spatio-temporal delays,} {J. Differ. Equations.}
{232} (2007) 104-133.

\bibitem{Brit} {N. F. Britton}, {Spatial structures and periodic traveling waves in an integro-
differential reaction-diffusion population model,} { SIAM J. Appl. Math.} {50} (1990) 1663-1688.

\bibitem{dk}{ O. Diekmann, H.G. Kaper,} {On the bounded solutions of a nonlinear convolution equation,} {Nonlinear Anal. TMA.} {2} (1978) 721-737.

\bibitem{fz}{J. Fang, X. Zhao,} {Existence and uniqueness of traveling waves for non-monotone integral equations with applications,} {J. Differ. Equations.} {248} (2010) 2199-2226.

\bibitem{FWZ}{J. Fang,  J. Wei, X. Zhao,} {Spatial dynamics of a nonlocal and time-delayed reaction-diffusion system,} {Journal of Differ. Equations.}
 {245} (2008) 2749-2770.
 
\bibitem{ft}{T. Faria, S. Trofimchuk,} Nonmonotone travelling
waves in a single species reaction-diffusion equation whith delay, {
J. Differential Equations.} {228} (2006) 357-376.

\bibitem{fhw} {  T. Faria, W. Huang, J. Wu,} Traveling
waves for delayed  reaction-diffusion equations with non-local
response, {Proc. Roy. Soc. London Sect. A.} {462} (2006) 229-261.

\bibitem{gpt} {C. Gomez, H. Prado, S. Trofimchuk, }Seperation dichotomy and wavefronts for a nonlinear convolution equation, {submited} (2012) arxiv.org/abs/1204.5760v1.

\bibitem{gouk} {S. A. Gourley, Y. Kuang}, {Wavefronts and global stability in time-delayed
population model with stage structure}, {Proc. R. Soc. A.} { 459} (2003) 1563-1579.

\bibitem{gouS} {S. A. Gourley, J. So,}{ Extinction and wavefront propagation in a
reaction-diffusion model of a structured population with distributed maturation
delay,} {Proc. Royal Soc. of Edinburgh.} {133A} (2003) 527-548.

\bibitem{gouss} {S. A. Gourley, J. So, J. Wu}, {Non-locality of reaction-diffusion equations induced by delay: biological modeling and nonlinear dynamics,} {J. Math. Sciences.} {124} (2004) 5119-5153.

\bibitem{lrw}{ W. T. Li, S. Ruan,  Z.C. Wang,} { On the diffusive Nicholsons blow�ies equation
with nonlocal delay,} {J. Nonlinear Science.} {17} (2007) 505-525.

\bibitem{LiWu} {D. Liang, J. Wu,}  Travelling waves and numerical approximations in a reaction-advection-diffusion equation with non-local delayed
effects, {J. Nonlinear Science.} {13} (2003) 289-310.

\bibitem{ma1} {S. Ma,} Traveling waves for non-local delayed diffusion equations
via auxiliary equations, {J. Differential Equations.} {237} (2007)
259-277.

\bibitem{MSo} {M. Mei, J. So,} {Stability of strong traveling waves for a non-local time-delayed reaction-diffusion equation,} {Proc. Roy. Soc. Edinburgh.} { A 138} (2008)  551-568.

\bibitem{OG}{J. Al-Omari, S.A. Gourley,} {Monotone wave-fronts in a structured population model with distributed maturation delay,} {IMA J. Appl. Math.} {70} (2005) 858-879.

\bibitem{swz}{ J. So, J. Wu, X. Zou,} {A reaction-diffusion model for a single species with age structure. I, Travelling wave fronts on unbounded domains,} {Proc. Roy. Soc.}  {457} (2001) 1841-1853.

\bibitem{tz} {H. R. Thieme, X.-Q. Zhao,} {Asymptotic speeds of spread and traveling waves
for integral equations and delayed reaction-diffusion models,}{ J. Differential
Equations.} {195} (2003) 430-470.

\bibitem{tat} {E. Trofimchuk, P. Alvarado,  S. Trofimchuk,} {On the geometry of wave solutions of a delayed reaction-diffusion
equation, }{J. Differential Equations.} {246} (2009) 1422-1444.

\bibitem{tt}{E. Trofimchuk, S. Trofimchuk,} {Admissible wavefront speeds for a single species reaction-diffusion equation with delay,} {Discrete Contin. Dyn. Syst.} {20} (2008) 407-423.

\bibitem{ttt} {E. Trofimchuk, V. Tkachenko, S. Trofimchuk, Slowly oscillating wave 
solutions of a single species reaction-diffusion equation with delay,} {J. Differential Equations.} {245} (2008) 2307-2332.

\bibitem{wlr2}  {Z.-C. Wang, W.T. Li,  S. Ruan,}
{Traveling Fronts in Monostable Equations with Nonlocal Delayed Effects,} {J. Dyn. Diff. Equat.} {20} (2008) 573-607.

\bibitem{wang} {H. Wang,} {On the existence of traveling waves for delayed 
reaction-diffusion equations,} {J. Differential Equations 247.} (2009) 887-905.

 \bibitem{m-p} {J. Mallet-Paret,} {The Fredholm alternative for functional differential equations of mixed type,} { J. Dynam. Differential Equations} { 11}  (1999) 1- 48. 


\bibitem{widder}{ D.V. Widder,} The Laplace Transform. (Princeton Mathematical
Series, no. 6.) Princeton University Press, 1941. 406 pp.

\bibitem{wl}{S. Wu, S. Liu,} {Uniqueness of non-monotone traveling wave for delayed reaction-diffusion equations,} {Applied Mathematics Letters.} {22} (2009) 1056-1061.

\bibitem{wl2}{S. Wu, S. Liu,} { Asymptotic speed of spread and traveling fronts for a nonlocal reaction-diffusion model with distributed delay,} {Applied Mathematical Modelling.} {33} (2009) 2757- 2765.

\bibitem{wl3}{S. Wu, S. Liu,} {Existence and uniqueness of  traveling waves for non-monotone integral  equations with application} {Journal of Mathematical Analysis and Applications.} {365} (2010) 729-741.

\bibitem{WM} J. Wu, D. Wei, M. Mei, Analysis on the critical speed of traveling waves, {Applied Mathematics Letters.} {20} (2007) 712-718.

\bibitem{xw} Z. Xu, P. Weng, Traveling waves for nonlocal and non-monotone delayed reaction-diffusion equations, {Acta Mathematca Sinica, English Series} http://maths.scnu.edu.cn/Uploadfiles/201352214355539.pdf.

\bibitem{xz} D. Xu, X. Zhao, Asymptotic speed of spread and traveling wave for nonlocal epidemic model, {Discrete  and Continuous Dynamical Systems-Series B.} {5} (2005) Number 4.

\bibitem{ycw} {T. Yi,Y. Chen, J. Wu,} Unimodal dynamical systems: Comparison principles, spreading speeds and traveling waves, {J.  Differential Equations.} {254} (2013)  3538-3572.



\end{thebibliography}
\end{document}